\documentclass[11pt]{amsart}
\usepackage[english]{babel}
\usepackage{amsmath,amsthm, amssymb,amscd, dsfont}
\usepackage{amsfonts}
\usepackage{xcolor}
\usepackage{cancel}
\usepackage{hyperref}
\usepackage{graphicx}
\newtheorem{thm}{Theorem}[section]
\newtheorem{cor}[thm]{Corollary}
\newtheorem{lem}[thm]{Lemma}

\newtheorem{prop}[thm]{Proposition}
\theoremstyle{definition}
\newtheorem{defn}[thm]{Definition}
\theoremstyle{remark}
\newtheorem{rem}[thm]{Remark}
\numberwithin{equation}{section}

\theoremstyle{remark}
\newcommand{\Comp}[2]{\Comp^{#1}_{#2}}

\renewcommand{\phi}{\varphi}

\newcommand\mathand{\,\&\,}

\DeclareMathOperator{\Ncl}{\mathtt{Ncl}}
\DeclareMathOperator{\Cj}{\mathtt{Cj}}

\newcommand{\rel}[1]{\mathrel{#1}}

\newcommand{\C}[2]{C^{#1}_{#2}}
\newcommand \la \langle
\newcommand \ra \rangle
\newcommand{\sss}{\sigma}
\newcommand{\ttt}{\tau}
\newcommand{\aaa}{\alpha}
\begin{document}

\title{Calibrating word problems of groups via  the complexity of
equivalence relations}

\author[Nies]{Andr\'e Nies}
\address{Department of Computer Science\\
University of Auckland\\
Private Bag
92019, Auckland, New Zealand}
\email{andre@cs.auckland.ac.nz}
\author[Sorbi]{Andrea Sorbi}
\address{Dipartimento di Ingegneria dell'Informazione e Scienze Matematiche\\
Universit\`a di Siena\\
53100 Siena\\
Italy}
\email{andrea.sorbi@unisi.it}
\thanks{
Nies is partially supported by the Marsden fund of New Zealand. Sorbi is a
member of INDAM-GNSAGA; he was partially supported by Grant 3952/GF4 of the
Science Committee of the Republic of Kazakhstan, and by PRIN 2012 ``Logica
Modelli e Insiemi ''} 
\date{}%

\begin{abstract}
(1) There is  a finitely presented group with a  word problem which is a
uniformly effectively inseparable equivalence relation. (2) There is  a
finitely generated  group of computable permutations with a word problem
which is  a universal co-computably enumerable equivalence relation. (3)
Each c.e.\  truth-table degree contains the word problem of a finitely
generated group of computable permutations.

\end{abstract}

\maketitle

\section{Introduction}
Given two equivalence relations $R, S$ on the set $\omega$ of natural
numbers, we say that $R$ is \emph{computably reducible to $S$}  (or, simply,
$R$ is \emph{reducible to $S$}; notation: $R \leq S$) if there exists a
computable function $f$ such that, for every $x,y \in \omega$,
\[
x \rel{R} y \Leftrightarrow f(x) \rel{S} f(y).
\]
The first systematic study of this reducibility on equivalence relations is
implicit in  Ershov~\cite{Ershov:NumberingsI, Ershov:NumberingsII}. Recently
this reducibility has been successfully applied to classify natural problems
arising in mathematics and computability theory: see for instance
in~\cite{Coskey-Hamkins-Miller, Fokina-Friedman-Nies, Fokina-et-al-several}.

In classifying objects according to their relative complexity, an important
role is played by objects that are universal, or complete, with respect to
some given class. We are interested in this notion for the case of
equivalence relations on $\omega$.

\begin{defn}
Let $\mathcal{A}$ be a class of equivalence relations. An equivalence
relation $R\in \mathcal A$  is called \emph{$\mathcal{A}$-universal}, (also
sometimes called $\mathcal{A}$-complete) if  $S \leq R$ for every $S \in
\mathcal{A}$.
\end{defn}
For instance, by~Fokina  et al.~\cite{Fokina-et-al-several}  the isomorphism
relation for various familiar classes of computable structures is
$\Sigma^1_1$-universal, and by~Fokina,  Friedman and Nies
\cite{Fokina-Friedman-Nies} the relation of computable isomorphism of c.e.\
sets is $\Sigma^0_3$-universal. Ianovski et al.\
\cite[Theorem~3.5]{Ianovski-et-al} provide a natural example  of a
$\Pi^0_1$-universal equivalence relation, namely equality of unary quadratic
time computable functions. In contrast, they show
\cite[Corollary~3.8]{Ianovski-et-al} that there is no $\Pi^0_n$-universal
equivalence relation for $n>1$.

In this paper we are interested in $\Sigma^{0}_{1}$-universal and in
$\Pi^{0}_{1}$-universal equivalence relations arising from group theory. They
arise naturally via word problems, if we view the word problem of a group as
the  equivalence relation that holds for two terms if they denote the same
group element.

In  Theorem~\ref{thm:main1}  we will build  a finitely presented group  with
a word problem as follows: each pair of distinct equivalence classes is  effectively
inseparable in a uniform way. Since this property for ceers implies
 $\Sigma^0_1$-universality (see \cite{ceers}), it follows that the word problem is
$\Sigma^0_1$-universal.

Finitely generated (f.g.) groups of computable permutations are special cases
of f.g. groups with a co-c.e. set of relators. The word problem of any
finitely generated (f.g.) group of computable permutations is $\Pi^0_1$.
Using the  theory of numberings, Morozov~\cite{morozov2000once} built an
example of a  f.g.\ group with $\Pi^0_1$ word problem  that  is not
isomorphic to  a f.g.\  group of computable permutations. (We conjecture that
future research might provide a natural example of such a group, generated
for instance by finitely many computable isometries of the Urysohn space.) As
our second main result, in Theorem~\ref{thm:main2} we will   build   a f.g.\
group of computable permutations  with a $\Pi^{0}_{1}$-universal word
problem. Thus, within the groups that have a $\Pi^0_1$ word problem, the
maximum complexity of the word problem is already assumed  within the
restricted class of f.g.\ groups of computable permutations. By varying the
methods, in Theorem~\ref{th:truth table} we show that every c.e.\ truth-table
degree contains the word problem of a 3-generated group of computable
permutations.

We include a number of open questions. Is the computably enumerable equivalence
relation of isomorphism among
finitely presented groups recursively isomorphic to equivalence of
sentences under Peano arithmetic?  What is the complexity of embedding and
isomorphism among f.g.\ groups of (primitive) recursive permutations? A
natural guess would be $\Sigma^0_3$ -universality.

\section{Background and preliminaries}
\subsection*{Group theory}
Group theoretic terminology and notations are standard, and can be found for
instance in \cite{Kargapolov-Merzljakov:Book}. Throughout let $F(X)$ be the
free group on $X$, consisting of all reduced words of letters from $X \cup
X^{-1}$, with binary operation induced by concatenation and cancellation of
$x$ with $x^{-1}$, and the empty string as identity; see
\cite[p.89]{Kargapolov-Merzljakov:Book} for notations and details. It is
customary to write $F(x_1, \ldots, x_k)$ if $X=\{x_1, \ldots, x_k\}$ is
finite. The symbol $\cong$ denotes isomorphism of groups, and, for a group
$H$ and a set $S \subseteq H$, by $\Ncl_{H}(S) $ one denotes the normal
closure of $S$ in $H$; if $H$ is clear from the context one writes $\Ncl(S)$.
A~\emph{presentation} of a group $G$ is a pair $\langle X; R\rangle$  with
$R\subseteq F(X)$ such that $G\cong {F(X)}{/\Ncl_{F(X)}(R)}$. It is
legitimate to write $G=\langle X; R\rangle$ since the presentation identifies
$G$ up to group isomorphism. The congruence corresponding to the normal
subgroup $\Ncl_{F(X)}(R)$ will be written as $=_{G}$; the relation $=_{G}$ is
clearly an equivalence relation on $F(X)$, which we will call the \emph{word
problem} of $G=\langle X; R \rangle$; the $=_{G}$-equivalence
class of an element $x$ will be denoted by $[x]_{G}$. If $X$ is a finite set then we can
encode the elements of $F(X)$ by natural numbers, and multiplication becomes
a binary computable function. A group $G=\langle X; R\rangle$ is
\emph{finitely presented} (f.p.) if both $X$ and $R$ are finite. It is easy
to see (under coding) that in this case, $=_{G}$ is a computably enumerable
equivalence relation on $\omega$.

Our terminology is slightly nonstandard because by the word problem of a
f.p.\ group $G=\langle X;  R\rangle$, one usually means the equivalence class
$[1]_{G}$ of the identity element $1$, and the problem of deciding, for a given word $w\in F(X)$, whether
$w \in [1]_{G}$. The difference is minor, though, since $=_{G}$ and the set
$[1]_{G}$ are $m$-equivalent. The $1$-reduction $x \mapsto \langle
x,1\rangle$ shows that $[1]_G \, \le_1 \, =_G$ (where the symbol $\le_{1}$
denotes $1$-reducibility), and the $m$-reduction $\langle x,y\rangle \mapsto
xy^{-1}$ shows that  $=_G \, \le_m \, [1]_G$ (where the symbol $\le_{m}$
denotes $m$-reducibility).

\subsection*{Effective inseparability}
The reader is referred to \cite{Soare:Book} for any unexplained notation and
terminology from computability theory. A partial computable function which
is total  is  simply called   a
computable function. If  $A, B
\subseteq \omega$, one writes  $A\equiv B$ if there exists a computable
permutation $f$ of $\omega$ such that $f(A)=B$;   if $(A,B)$ and $(C,D)$ are
disjoint pairs of subsets of $\omega$, one writes  $(A,B) \equiv (C,D)$, if
there exists a computable permutation $f$ of $\omega$ such that $f(A)=C$ and
$f(B)=D$. We recall that a  disjoint pair of sets $(A,B)$ is called
\emph{recursively inseparable} if there is no recursive set  $X$ such that $A
\subseteq X$ and $B \subseteq X^c$, where $X^{c}$ denotes the complement of
$X$. The following property is stronger: $(A,B)$  is \emph{effectively
inseparable} (\emph{e.i.}) if there is \emph{productive} function, that is, a
partial computable function $\psi(u,v)$ such that
\[
(\forall u,v)[A \subseteq W_u \mathand B \subseteq W_v \mathand W_u \cap W_v =
\emptyset \Rightarrow \psi(u,v)\downarrow \notin W_u \cup W_v].
\]

\begin{rem}\label{stuff-on-ei}
It is well known (see e.g.\ \cite[II.4.13]{Soare:Book}) that if $(A,B)$ and
$(C,D)$ are disjoint pairs of c.e.\ sets then:
\begin{enumerate}
\item[-] $(C,D)$ e.i.\ implies $(A,B) \le_1 (C,D)$;
\item[-] if both pairs are e.i.\ then $(A,B) \equiv (C,D)$;
\item[-] if $(A,B) \le_m (C,D)$ and $(A,B)$ is e.i.\ then $(C,D)$ is e.i.\
    as well;
\item[-] if $A \subseteq C$, $B \subseteq D$ and $(A,B)$ is e.i.\ then
    $(C,D)$ is e.i.\ as well.
\end{enumerate}
 \end{rem}
The following fact  about e.i.\ pairs of c.e.\  sets will be used in the
proof of Theorem~\ref{thm:main1}.

\begin{lem}\label{lem:uei-product}
If $(A,B)$ and $(C,D)$ are e.i.\ pairs of c.e.\ sets, then so is
the pair $(A\times C, B\times D)$. Moreover,   a productive function for
$(A\times C, B\times D)$ can be found uniformly from productive functions for
$(A,B)$ and $(C,D)$.
\end{lem}

\begin{proof}
We prove in fact that if $(A,B)$ is a disjoint pair of c.e.\  sets, and
$(C,D)$ is e.i., then $(A,B) \le_{1} (A\times C, B\times D)$: hence, if
$(A,B)$ is e.i., then $(A\times C, B\times D)$ is e.i.\  as well. Let $g$ be
a computable function such that $g(A)\subseteq C$ and $g(B)\subseteq D$; such
a function exists because $(A,B) \le_{1} (C,D)$.   Clearly the
$1$-$1$ computable function
\[
f(x)=\langle x, g(x)\rangle
\]
provides a $1$-reduction showing that $(A,B) \le_{1} (A\times C, B\times D)$.

The claim about uniformity is straightforward.
\end{proof}

  Although not used in this paper, it is worth noting that a statement
analogous to the   lemma  above holds  when we replace ``effectively
inseparable'' by the weaker notion of being recursively inseparable.

\begin{prop}\label{lem:rec-ins}
If $(A,B)$ and $(C,D)$ are recursively inseparable pairs of c.e.\  sets,
then so is $(A\times C, B\times D)$.
\end{prop}

\begin{proof}
Assume that $R$ is a computable set such that $A\times
C \subseteq R$ and $B\times D \subseteq R^{c}$. For every
$v$, let
\[
R_{v}=\{x: \langle x, v\rangle \in R\}.
\]
We observe that for every $v$ there exists $x \in A$ such that $\langle x,
v\rangle \in R^{c}$, or there exists $x \in B$ such that $\langle x, v
\rangle \in R$;  otherwise $A\subseteq R_{v}$ and $B \subseteq R_{v}^{c}$,
which would contradict the inseparability of $(A,B)$. Let $R_{A} $ and $R_{B}
$ be computable binary relations such that
\begin{align*}
&(\exists x)[x \in A \,\&\,\langle x,v\rangle \in R^{c}] \Leftrightarrow
            (\exists s)R_{A}(v,s),\\
&(\exists x)[x \in B \,\&\,\langle x,v\rangle \in R ] \Leftrightarrow
            (\exists s)R_{B}(v,s),
\end{align*}
and define
\[
U=\{v: (\exists s)[R_{A}(v,s)] \,\&\, (\forall t \le s) \neg R_{B}(v,t)]\}.
\]
The set $U$ is decidable, as we have seen that for every $v$, there exists $x
\in A$ such that $\langle x,v\rangle \in R^{c}$, or there exists $x \in B$
such that $\langle x,v\rangle \in R$. Now $v \in C \cap U$ implies  $(\exists
x)[x \in A \,\&\, \langle x,v\rangle \in R^{c}]$ contrary to $A \times C
\subseteq R$. Similarly, $v \in D \smallsetminus U$ implies $(\exists x)[x
\in B \,\&\,\langle x,v\rangle \in R]$, contrary to  $B \times D \subseteq
R^c$. We conclude  that $C \subseteq U^{c}$ and $D \subseteq U$, which is the
final contradiction.
\end{proof}

\subsection*{C.e.\  equivalence relations and word problems}
Computably enumerable equivalence relations have been  studied extensively;
see for instance \cite{Bernardi-Sorbi:Classifying, Ershov:positive,
Gao-Gerdes}. While they are called \emph{positive} in the Russian literature,
we call such an equivalence relation a \emph{ceer} following Andrews et al.\
\cite{ceers}. $\Sigma^{0}_{1}$-universal ceers arising naturally in formal
logic have been pointed out for instance in~\cite{Bernardi-Montagna:extensional, Montagna:ufp,
Visser:Numerations}.

\begin{defn}[\cite{Bernardi:the-relation-provable}]\label{uei eqrel}
A ceer $E$ is called \emph{uniformly effectively inseparable} (\emph{u.e.i.})
if there is  a computable  binary function $p$ such that, whenever
$a\cancel{\rel{E}} b$, the partial computable function $\psi(u,v)
=\phi_{p(a,b)}(u,v)$ witnesses that the pair of equivalence classes $([a]_E,
[b]_E)$ is e.i.\
\end{defn}
\noindent As already observed in the introduction, it is shown in
\cite{ceers} that every u.e.i.\ ceer is $\Sigma^{0}_{1}$-universal. It is
worth recalling that uniformity plays a crucial role in yielding
universality, as there are non-universal ceers yielding a partition of
$\omega$ into effectively inseparable pairs of distinct classes~\cite{ceers}.

Surprisingly, f.p.\ groups  with a $\Sigma^0_1$-universal word problem
appeared in the literature prior to any explicit study of computable
reducibility among equivalence relations.  Charles F.\
Miller~III~\cite{MillerIII-decision} proved that there exists a f.p.\  group
with $\Sigma^0_1$-universal word problem.  He  shows that  another
interesting equivalence relation is $\Sigma^0_1$-universal:  the isomorphism
relation between finite presentations of groups, which (via encoding of
finite presentations  by numbers)  can be seen as a ceer.  Not knowing  of
this much earlier result, Ianovski, Miller, Ng, and  Nies~\cite[Question
6.1]{Ianovski-et-al} had recently posed this as an open question.

\begin{thm}[\cite{MillerIII-decision}]\label{thm:MillerIII} \mbox{}
\begin{enumerate}
  \item  Given a ceer $E$  one can effectively build a f.p.\ group
      $G_{E}=\langle X; R \rangle$, and a computable sequence of words
      $(w_{i})_{ i \in \omega}$ in $ F(X)$ such that, for every $i,j$,
      \[
i \rel{E} j \Leftrightarrow w_{i}=_{G_{E}} w_{j}.
\]
  \item Given  a finite presentation $\langle X; R \rangle$ of a group $G$
      one  can effectively find a computable family  $(H^G_w )_{w  \in
      F(X)}$ of f.p.\ groups such that, for all $v, w \in F(X)$,
\[
v =_G w \Leftrightarrow H^G_{v} \cong H^G_{w}.
\]
\end{enumerate}
\end{thm}
\begin{proof}
The first item is obtained  in \cite[p~90f]{MillerIII-decision}, used as a
preliminary step to prove Theorem~V.2. The second item is
\cite[Theorem~V.1]{MillerIII-decision}.
\end{proof}

\begin{cor}  \label{cor:Miller} \mbox{}
\begin{enumerate}
\item There exists a f.p.\ group $G$ such that $=_{G}$ is a
    $\Sigma^{0}_{1}$-universal ceer.
\item The isomorphism problem  $\cong_{f.p.}$ between finite presentations
    of groups is a $\Sigma^{0}_{1}$-universal ceer.
\end{enumerate}

\end{cor}

\begin{proof}
Let $E$ be a $\Sigma^{0}_{1}$-universal ceer. Then
\begin{enumerate}
  \item by Theorem~\ref{thm:MillerIII}(1), $E\le \, =_{G_E }$, and thus
      $=_{G_{E}}$ is $\Sigma^{0}_{1}$-universal;
  \item by Theorem~\ref{thm:MillerIII}(2),
\[
i \rel{E} j \Leftrightarrow H^{G_{E}}_{v} \cong H^{G_{E}}_{w}.
\]
This shows that $E \le \, \cong_{f.p.}$, whence $\cong_{f.p.}$ is
$\Sigma^{0}_{1}$-universal.
\end{enumerate}
\end{proof}

We observe   that $\Sigma^0_1$-universality of the word problem  does not
necessarily imply being u.e.i.

\begin{thm}\label{thm:notei}
There exists a f.p.\ group $G$ such that $=_G$ is $\Sigma^0_1$-universal, but
not u.e.i.
\end{thm}

\begin{proof}
We build  a f.p.\ group $G$ such that $=_G$ is $\Sigma^0_1$-universal, but it
does not even yield a partition into recursively  inseparable pairs of
disjoint equivalence classes. To see this, let $H = \langle X; R \rangle$ be
a f.p.\  group such that $=_H$ is $\Sigma^0_1$-universal. Let $v\not \in X$
be a new letter. The free product $G=H \ast F(v)$ (where $F(v)$ is the free
group on $v$) has the finite presentation $\langle X, v; R\rangle$. Since $H$
can be seen as a subgroup of $G$ and the embedding is computable, the group
$G$ has $\Sigma^0_1$-universal word problem. Any word $w \in F(X \cup \{v\})$
can be uniquely written as $w=h_1v^{n_1}h_2\cdots v^{n_r}h_{r+1}$, with $h_j
\in F(X)$ and $n_j \neq 0$, for all $j$. Let
\[
n_v(w)=n_1+ \cdots +n_r
\]
be the exponent sum of $v$ in $w$, and let $S=\{w \in F(X \cup \{v\}):
n_v(w)=0\}$. It is immediate that $[1]_G\subseteq S$ and $[v]_G \subseteq
S^c$, so  the recursive set $S$ separates the pair $([1]_G, [v]_G)$.
\end{proof}

The proof of the previous theorem suggests an additional comment. We observe
that if in a group $G$ the operations are computable, then all
$=_G$-equivalence classes are uniformly computably isomorphic: the function
$w \mapsto wu^{-1}v$ is a computable permutation of the group (uniformly
depending on $u,v$) which maps $[u]_G$ onto $[v]_G$. Thus if an equivalence
class $[u]_G$ is creative, so is any other equivalence class $[v]_G$, and
creativeness holds uniformly, i.e. there is a computable function $p$ such
that, for every $v$, $\phi_{p(v)}$ is productive for the complement of
$[v]_G$. Nothing like this holds for effective inseparability, or for
computable inseparability. Indeed, one can take the group $H$ considered in
the proof of Theorem~\ref{thm:notei} to be such that its word problem yields
at least a pair of effectively inseparable classes (for instance take $H=D$,
where $D$ is the group built in Theorem~\ref{thm:main1} in which all distinct
pairs of equivalence classes are effectively inseparable). Thus the word
problem of the group $G$ of Theorem~\ref{thm:notei} does have effectively
inseparable classes, but not all pairs are so, since there are pairs which
can be computably separated.

\section{A finitely presented group with u.e.i.\ word problem}
We now build   a f.p.\  group  with a  word problem that is a u.e.i.\ ceer.
We first provide   Lemma~\ref{lem:normal-closure}   that if $G$ is a f.p.\
group containing a word $w$ such that $([1]_G, [w]_G)$ is e.i., then all
disjoint pairs $([s]_G, [t]_G)$ with $s, t \in \Ncl_G(w)$ are   e.i.\ in a
uniform way. For the main construction, using a result of Miller III, we take
a computably presented group~$A$ containing a word~$w$ such that the pair
$([1]_A, [w]_A)$ is e.i. By the Higman Embedding Theorem combined with a
construction due to Rabin, we embed $A$ into a f.p.\ group $D$ so that if
$N$ is a non-trivial normal subgroup of $D$, with $w \in N$, then $N=D$.
Taking $N=\Ncl_D(w)$
and observing that the pair $([1]_D, [w]_D)$ is also  e.i., the lemma shows
that $=_D$ is u.e.i.

\begin{lem}\label{lem:normal-closure}
Let $G=\langle X; R\rangle$ be a given f.p.\ group, and let $w$  be an element
of $F(X)$ such that  $([1]_G, [w]_G)$ is e.i. Let $N=\Ncl_G(w)$. For   $s,t
\in N$ such that  $s \ne_G t$, the pair of sets $([s]_G,[t]_G)$ is e.i.\
uniformly in $s,t$.
\end{lem}

\begin{proof}
Since $([s]_G, [t]_G)\equiv ([1]_G, [s^{-1}t]_G)$, it suffices to show that
$([1]_G, [r]_G)$ is uniformly e.i.\ for any $r \in N \smallsetminus [1]_G$.
Note that $N$ consists of the products of conjugates of $w$ and of $w^{-1}$,
so  it is enough to show:
\begin{enumerate}
  \item if $([1]_G, [u]_{G})$ is e.i., then so is $([1]_G, [u^{-1}]_{G})$:
      this follows from the fact that $([1]_{G}, [u]_{G}) \equiv
      ([u^{-1}]_{G}, [1]_{G},)$, via the computable permutation $x \mapsto
      u^{-1} x$;
  \item if $([1]_G, [u]_{G})$ is e.i., then so is $([1]_{G}, [g^{-1}u
      g]_{G})$ for every $g \in G$: the computable permutation $x \mapsto
      g^{-1} x g$ provides an isomorphism $([1]_G, [u]_{G})\equiv
      ([1]_{G}, [g^{-1}u g]_{G}) $;
  \item if $uv \ne_G 1$ and the pairs $([1]_{G}, [u]_{G})$ and $([1]_{G},
      [v]_{G})$ are e.i., then $([1]_{G}, [uv]_{G})$ is e.i.:  \\
      By Lemma~\ref{lem:uei-product} the pair $([1]_G \times [1]_G,
      [u]_G\times [v]_G)$ is e.i. On the other hand, let
      \begin{align*}
       X&=\{\langle w,z\rangle : wz \in [1]_G\},\\
       Y&=\{\langle w,z\rangle : wz \in [uv]_G\}.
      \end{align*}
Then    $[1]_G \times [1]_G \subseteq X$ and $ [u]_G\times [v]_G\subseteq
      Y$, and thus, by Remark~\ref{stuff-on-ei}, $(X, Y)$ is e.i. Since
      $(X,Y)\le_m ([1]_{G}, [uv]_{G})$ via the mapping $\langle w, z\rangle
      \mapsto wz$, it follows that $([1]_{G}, [uv]_{G})$ is e.i., as
      desired.
\end{enumerate}
Each step provides being e.i.\ in a uniform fashion. If $r\in N$ we can
obtain its  representation as a product of conjugates of $w$ and of $w^{-1}$
effectively. Since $[1]_G$ and $N$ are c.e., there is a partial computable
function $p$ such that $\phi_{p(a,r)}$ is productive for $([a]_G, [r]_G)$,
when $a \in [1]_G$ and $r \in N \smallsetminus [1]_G$. So $([1]_G,[r]_G)$ is
e.i.\ uniformly in $r$, whence $([s]_G,[t]_G)$ is e.i.\ uniformly in $s,t$
 as required.
 \end{proof}

\begin{thm}\label{thm:main1}
There exists a f.p.\  group $D$ such that $=_D$ is u.e.i.
\end{thm}

\begin{proof}
For elements $u,t$ of a group, we write $\Cj(u,t)=t^{-1}ut$. Following
\cite{MillerIII-quotients},  take an e.i.\ pair $(Y_0,Y_1)$ of c.e.\ sets.
Let
$F=F(c,d)$ be the free group on two generators $c,d$; for every $i>0$, let
\[
b_{i-1}=\Cj(\Cj(c,d^{-1}),c^{i})\cdot \Cj(\Cj(\Cj(c^{-1},d),c^{i}),d^{-2}).
\]
Next let
\[
R=\Ncl_{F}(\{b_{0}b_{i}^{-1}: i \in Y_0\} \cup \{b_{1}b_{j}^{-1}: j \in Y_1\}),
\]
and let  $A=\langle c,d; R \rangle$.
Note
that $A$ is   a \emph{computably presented}  group, namely $A$ has a
presentation $\langle Z; T\rangle$ where $Z$ is finite and $T$ is c.e. It can
be shown \cite{MillerIII-quotients} that the computable mapping $i \mapsto
b_i$ provides a reduction
\[
(Y_0,Y_1) \le_1 ([b_0]_A, [b_1]_A).
\]
Hence, by the third item in Remark~\ref{stuff-on-ei}, the pair  $([b_0]_A,
[b_1]_A)$ is e.i. We now follow  a line of argument as in the proof of
Theorem~IV.3.5 of \cite{Lyndon-Schupp}, to which the reader is referred to
fill in the details of the present proof;  the only difference between our
proof and that in~\cite{Lyndon-Schupp} is that we first embed $A$ into a
f.p.\ group $L$, aiming at a final f.p.\ group $D$, whereas in the proof of
Theorem~IV.3.5 of \cite{Lyndon-Schupp} the starting group $C$ is first
embedded into a countable simple group $S$, as the goal in that case is to
end up with a finitely generated simple group. (The construction provided by
Theorem~IV.3.5 of \cite{Lyndon-Schupp} is due to Rabin~\cite{Rabin}; the
version presented in \cite{Lyndon-Schupp} is modelled on
Miller~III~\cite{MillerIII-decision}.)

By the Higman Embedding Theorem (\cite{Higman}; see also
\cite[Theorem~IV.7.1]{Lyndon-Schupp}) the computably presented group $A$ can
be embedded into  a f.p.\ group $L$; next embed, using
\cite[Theorem~IV.3.1]{Lyndon-Schupp},  the free product $L \ast F(x)$ (with
$x$ a new generator) in a f.p.\ group $U$, generated by $u_1$ and $u_2$ both
of infinite order.

In order to build the desired f.p.\ group $D$, we are now going to introduce
additional groups, using two well known combinatorial group theoretic
constructions, namely HNN-extension (where HNN stands for
Higman-Neumann-Neumann), and free product with amalgamation. We briefly
recall these two constructions. If $G=\langle T; Z\rangle$ is a group
presentation, and $\phi:H\rightarrow K$ is an isomorphism between subgroups
of $G$, then the \emph{HNN-extension of $G$, relative to $H,K$ and $\phi$},
is the group $\langle T, p; Z \cup \{p^{-1}h p=\phi(h): h \in H\}\rangle$, of
which $G$ is a subgroup, and $p$  (with $p \notin G$) realizes by conjugation
the given isomorphism; $p$ is called the \emph{stable letter}. It is clear
that one can limit oneself to let the added relations vary on a set of
generators of $H$, instead of adding one relation for each $h \in H$.
Moreover, if $G_{1}=\langle T_{1}; Z_{1}\rangle$, $G_{2}=\langle T_{2};
Z_{2}\rangle$ are group presentations of disjoint groups, with two isomorphic
subgroups $H_{1}, H_{2}$, via isomorphism $\phi: H_{1}\rightarrow H_{2}$,
then their \emph{free product amalgamating $H_1$ and $H_2$ by $\phi$} is the
group $\langle T_{1}\cup T_{2}; Z_{1}\cup Z_{2} \cup \{h=\phi(h): h \in
H_{1}\}\rangle$, which is intuitively the ``freest'' overgroup of both $G_{1}$
and $G_{2}$ in which their subgroups are identified. Again,  it is clear that
one can limit oneself to let the added relations vary on a set of generators
of $H_{1}$, instead of adding one relation for each $h \in H_{1}$. For more
on these constructions, see \cite{Lyndon-Schupp}.

Consider the groups
\begin{align*}
J&=\langle U, y_1, y_2; y_1^{-1}u_1 y_1=u_1^2, y_2^{-1}u_2 y_2=u_2^2\rangle,\\
K&= \langle J,z; z^{-1}y_1z=y_1^{2}, z^{-1}y_2z=y_2^{2}\rangle,\\
P&=\langle r,s; s^{-1}rs=r^{2}\rangle,\\
Q&=\langle r,s,t; s^{-1}rs=r^{2}, t^{-1}st=s^{2}\rangle.\\
\end{align*}
The group $J$ is the (double) HNN-extension of $U$ with stable letters
$y_{1}, y_{2}$, where for each $i\in \{1,2\}$, $y_{i}$ realizes by
conjugation the isomorphism induced by $u_{i} \mapsto u_{i}^{2}$, between the
subgroups generated by $u_{i}$, and  by $u_{i}^{2}$, respectively; $K$ is the
HNN-extension of $J$, with stable letter $z$, realizing by conjugation the
isomorphism induced by $y_{1} \mapsto y_{1}^{2}$ and $y_{2} \mapsto
y_{2}^{2}$, between the subgroups generated by $y_{1}, y_{2}$,  and by
$y_{1}^{2}, y_{2}^{2}$, respectively; $P$ is the HNN-extension of $F(r)$,
with stable letter $s$, realizing by conjugation the isomorphism  induced by
$r \mapsto r^{2}$, between the subgroups generated by $r$, and  by $r^{2}$,
respectively; $Q$ is the HNN-extension of $P$, with stable letter $t$,
realizing by conjugation the isomorphism induced by $s \mapsto s^{2}$,
between the subgroups generated by $s$, and by $s^{2}$, respectively. It is
shown in the proof of \cite[Theorem~IV.3.4]{Lyndon-Schupp} that $r,t$ freely
generate a   subgroup of $Q$. Let $w \in L$, with $w \ne_L 1$: since the
commutator $[w,x]$ has infinite order in $U$, an argument similar to the one
used for $r,t$, and $Q$ (see again \cite{Lyndon-Schupp}) shows that $z$ and 
$[w,x]$ freely generate a  subgroup of $K$. Thus, one can form the free
product with amalgamation
\[
D=\langle K \ast Q; r=z, t=[w,x]\rangle.
\]
All  groups mentioned are finitely presented except for $A$.
We summarize in the following diagram the chains of embeddings provided by the constructions:
\[
\begin{CD}
A @>>> L @>>> L \ast F(x) @>>> U @>>> J @>>> K\\
   @.           @.                         @.        @.         @.     @VVV\\
   @.          @. F(r)            @>>> P @>>> Q @>>> D.
\end{CD}
\]
As pointed out in the proof of \cite[Theorem~IV.3.4]{Lyndon-Schupp}, if $N
\lhd D$ and $w \in N$, then $w=1$ in the quotient $D/N$. Then $[w,x]=1$ in
this quotient. Using the relators, we conclude that $t=1$, $s=1$, $r=1$,
$z=1$, $y_1=1$, $y_2=1$, $u_1=1$ and $u_2=1$. Therefore the quotient is
trivial, and hence $N=D$.

Keeping track of the images of the generators $c,d$ of $A$ into $D$, under
the chain of embeddings leading from $A$ to $D$, one sees that there is a
computable function $k$ from $F(c,d)$ into $F(X)$, where $X$ is the set of
generators of $D$ in the exhibited presentation of $D$, inducing the embedding
of $A$ into $D$. Let us identify
$k(a)$ with $a$, for all $a \in F(c,d)$. Since, under this identification, $b_{0}\neq_{D}
b_{1}$,  $[b_0]_A \subseteq [b_0]_D, [b_1]_A \subseteq  [b_1]_D$, and $([b_0]_A,
[b_1]_A)$ is e.i.,  it follows that  $([b_0]_D, [b_1]_D)$ is e.i.\ by the
last item in Remark~\ref{stuff-on-ei}. Let $w=b_1^{-1}b_0$: then
$w \ne_D 1$, the pair $([1]_D, [w]_D)$ is e.i., and by
Lemma~\ref{lem:normal-closure} the normal closure $N=\Ncl_D(w)$ satisfies
the property that all pairs $([s]_{D}, [t]_{D})$ of disjoint equivalence
classes of $N$ are e.i., uniformly in $s,t$. Since $w \in N$, it follows that
$N=D$. Therefore $D$ is a f.p.\ group with u.e.i.\ word problem.
\end{proof}

\section{Diagonal functions}
A \emph{diagonal} function for an equivalence relation $E$ is a computable
function $\delta$ such that $a \cancel{\rel{E}} \delta(a)$, for all $a$. In
this section we apply diagonal functions to   ceers arising from group
theory, and pose some related open questions. Following~\cite{Montagna:ufp},
a ceer $E$ is \emph{uniformly finitely precomplete}  if there exists a
computable function $f(D,e,x)$ such that 
\[
\phi_e(x)\downarrow \in [D]_E \Rightarrow f(D,e,x) \rel{E} \phi_e(x),
\]
for all $D,e,x$, where $D$ is a finite set and
$[D]_E$ denotes the $E$-closure of $D$.  (Here, and in the following, when given as an input to a computable
function, a finite set will be always identified with its canonical index.) An important example of a uniformly
finitely precomplete ceer is provable equivalence in Peano Arithmetic, i.e.\
the ceer $\sim_{PA}$ defined by $\ulcorner \sigma \urcorner \sim_{PA}
\ulcorner \tau \urcorner$ if and only if $\vdash_{PA} \sigma
\leftrightarrow \tau$.  Here $\sigma, \tau$ are sentences of $PA$, and we
refer to some computable bijection $\ulcorner \mbox{} \urcorner$ of the set
of sentences with $\omega$. A diagonal function is given by $\delta(\sigma) =
\neg \sigma$.

Ceers $E$ and $F$ are  called \emph{computably isomorphic} if there exists a
computable permutation $p$ of $\omega$ such that $p(E) = F$. The notions of a
diagonal function and a uniformly finitely precomplete  ceer play an
important role in the study and classification of $\Sigma^{0}_{1}$-universal
ceers.

\begin{prop}[\cite{Montagna:ufp}]
(i) Every uniformly finitely precomplete ceer is u.e.i.

\noindent (ii) A ceer $E$ is computably isomorphic to $\sim_{PA}$ if and only
if $E$ is uniformly finitely precomplete and $E$ has a diagonal function.
\end{prop}

A  \emph{strong diagonal} function for an equivalence relation $E$ is a
computable function $\delta$ such that $\delta(D) \notin [D]_E$, for every
finite set $D$.  Andrews and Sorbi~\cite{jumpsofceers} have shown   that
every u.e.i.\ ceer with a strong diagonal function is uniformly finitely
precomplete, and therefore computably isomorphic to $\sim_{PA}$.

Suppose a f.p.\ group $G = \langle X; R \rangle$ is nontrivial, say $w\ne_G
1$ for some $w \in F(X)$. Then $=_{G}$ has a diagonal function, namely the
map $\delta(r) = rw$ ($ r \in F(x)$). It would be interesting to prove that
there exists a f.p.\ group $G$ such that $=_G$ is uniformly finitely
precomplete, for this would yield an example of a word problem of a f.p.\
group which is computably isomorphic to $\sim_{PA}$. To show this, one can
try to strengthen Theorem~\ref{thm:main1} to provide a f.p.\ group $G$ such
that  $=_G$ is u.f.p., or, equivalently, to extend its proof in order to
provide  a f.p.\ group $G$ such that $=_G$ is u.e.i.\  and $G$ has a strong
diagonal function. Thereafter one can  use the above-mentioned result of
Andrews and Sorbi~\cite{jumpsofceers}. We do not know at present how to do
carry out this plan.

\begin{prop}
The isomorphism problem  $\cong_{f.p.}$ between finite presentations of
groups has a strong diagonal function.
\end{prop}
\begin{proof}
Uniformly in  a finite presentation $G = \langle x_1, \ldots, x_n ; r_1,
\ldots, r_k\rangle$,  the abelianization $G_{ab}$ has the finite presentation
\[
G_{ab} = \langle x_1, \ldots, x_n ; r_1, \ldots, r_k, [x_i,x_j]:
1 \le i<j \le n \rangle,
\]
where $[u,v]=u^{-1}v^{-1}uv$ is the usual commutator of $u,v$.
Given a finite set $S = \{G_1, \ldots, G_r\}$ of finite presentations, let
$\delta(S)$ be the canonical finite presentation of  the abelian group
$H=\mathds{Z} \times \prod_{1\le u \le r} (G_u)_{ab}$. Then $H \not \cong G_u$
for each $u$. For, if $G_u$ is abelian, then the torsion free rank of $H$
exceeds that of $G_u$.

We note that, via a   less elementary method involving  the Grushko-Neumann
Theorem (see~\cite[p.\ 178]{Lyndon-Schupp}), one could also  simply let  $H$
be the amalgam of $\mathds{Z}$ and all the $G_u$.
\end{proof}
We conjecture that $\cong_{f.p.}$ is uniformly finitely precomplete, and
hence computably isomorphic to $\sim_{PA}$.  In view of the foregoing
proposition it would suffice to show that the ceer $\cong_{f.p.}$ is  u.e.i.
By a result of Rabin~\cite{Rabin},  every equivalence class of $\cong_{f.p.}$ is creative;
see also \cite[p.\ 79]{MillerIII-decision}.

\section{$\Pi^{0}_{1}$-universality and groups of computable permutations}
We use the following notation: the product $\alpha\beta$ of two permutations
on some set $S$  is the permutation $\alpha\beta(s)=\beta(\alpha(s))$ where
$s \in S$.

\begin{thm}\label{thm:main2}
 There is a f.g.\ group of computable permutations  with a
$\Pi^{0}_{1}$-universal word problem.
\end{thm}

\begin{proof}
Given a $\Pi^0_1$ equivalence relation $E$, by \cite[Prop.\
3.1]{Ianovski-et-al} there is a computable binary function $f$ such
that
\[
x\rel{E}y \Leftrightarrow (\forall n) [f(x,n) = f(y,n)].
\]
The construction of $f$
shows that   $f(x,n)\le x$ for each $x,n$.

Fix now a $\Pi_{1}^{0}$-universal equivalence relation $E$ 
(for the existence of such an equivalence relation see~\cite{Ianovski-et-al}) and a corresponding function
$f$ as above. Via  a computable bijection  we identify $\mathds{Z} \times \omega$ with
$\omega$. We think of the domain of our computable permutations as a
disjoint union of pairs of ``columns" $$C^i_x = \{ 2x+i\} \times \omega,$$
where $i =0,1$, $x \in \mathds{Z}$ for the rest of this proof.

The first two of the three computable permutations $\sigma, \tau, \alpha$ we
are  about to define do not depend at all on $f$. The permutation $\sigma$
shifts $\C i x $ to $\C i {x+1}$:
\[
\sss (\la 2x+i, n \ra) = \la 2x+2+i, n \ra.
\]
The permutation $\ttt$ exchanges $\C i 0 $ with  $\C {1-i} 0$ and is the
identity elsewhere:
\[
\ttt (\la  i, n \ra) = \la  1-i, n \ra.
\]
We now define a computable permutation $\aaa$ coding $f$ in the sense that
there exists a fixed computable sequence $(t_x(\alpha, \sigma,\tau))_{x \in
\omega}$ of words in the free group generated by the symbols $\aaa, \sss, \ttt$, such that
for each $x,y \in \omega$,
\begin{equation} \label{eqn: code}
 \forall n \, f(x,n) = f(y,n)  \Leftrightarrow t_x
= t_y,
\end{equation}
where equality $t_x= t_y$ is in the group
generated by the three permutations.
For each $x, n$, the permutation $\aaa$ has a cycle of length $f(x,n)$ in the
interval $n(x+1), \ldots, (n+1)(x+1) -1$ of $C_{x}^{0}$. Thus, for each $x,n
\in \omega$ and $k \le x$,
\[
\aaa(\la 2x, n(x+1) +k \ra) =
\begin{cases}    \la 2x, n(x+1) + k+1 \ra & \mbox{if} \ k< f(x,n) \\
					  \la 2x, n(x+1)  \ra & \mbox{if} \ k= f(x,n) \\
					  \la 2x, n(x+1) +k  \ra & \mbox{otherwise,}
\end{cases}
\]
and $\alpha$ is the identity on the remaining columns.
We now define the terms $t_x$ for $x \in \omega$. The permutation $t_x
(\aaa, \sss, \ttt)$ will only retain the encoding of the values $f(x,n)$, and
erase all other information. It also moves this information to the pair of
columns $\C 0 0, \C 10$. In this way  we can compare the values $f(x,n)$ and
$f(y,n)$ applying  $t_x$ and $t_y$ to $\aaa, \sss, \tau$.

Recall that for elements $u,t$ of a group we write $\Cj(u,t)=t^{-1}ut$.
We let
\[
t_x=\Cj(\alpha, \sigma^{-x})\tau \Cj(\alpha^{-1}, \sigma^{-x}).
\]
Let $\alpha_x$   be the permutation    given by $\alpha(\langle 2x,
y\rangle)=\langle 2x, \alpha_x(y)\rangle$. Using that everything cancels
except what $\alpha$ codes on the column $C^0_x$,  we obtain
\[
t_x(\langle u,y\rangle)=
\begin{cases}
\langle u,y \rangle, &\text{if $u\ne 0,1$},\\
\langle 1, \alpha_x(y)\rangle, &\text{if $u=0$},\\
\langle 0, (\alpha_x)^{-1}(y)\rangle, &\text{if $u=1$}.
\end{cases}
\]
By the definition of $\aaa$ it is now clear that (\ref{eqn: code}) is
satisfied, and thus our $\Pi^{0}_{1}$-universal $E$ is reducible to
the word problem of $G$.
\end{proof}
In the area of  computational complexity, one writes input numbers in binary
and considers time bounds compared to their length. A quadratic time variant
$G$ of the function $f$  encoding the equivalence relation $E$ is obtained in
\cite[Theorem~3.5]{Ianovski-et-al}.  Some modifications to the proof above
yield three permutations that are polynomial time computable, as are their
inverses, and they  still  generate a group with $\Pi^0_1$-universal word
problem.

Independently Fridman~\cite{Fridman1962}, Clapham~\cite{clapham1964finitely}
and Boone~\cite{Boone1966a,Boone1966b,Boone1971} proved that each c.e.\
Turing degree contains the word problem of a f.p.\ group. (Here and
throughout next theorem and its proof, ``word problem'' is meant classically
as the equivalence class of the identity element). Later
Collins~\cite{Collins-tt-1971} extended this to c.e.\ truth table degrees. In
contrast, Ziegler~\cite{Ziegler1976}  constructed a bounded truth-table
degree that does not contain the word problem of a f.p.\ group. For f.g.
groups with $\Pi^{0}_{1}$ word problem, Morozov~\cite{morozov2000once} has
shown that there is a two-generator group which is not embeddable into  the
group of computable permutations of $\omega$.

 Using the
methods of the foregoing result, here we obtain an analog of the results by
Fridman, Clapham, Boone and Collins  for f.g.\ groups of computable
permutations. In fact we can choose the permutations of a special kind.

Let us call a permutation $\sss$ \emph{fully primitive recursive} if both
$\sigma $ and $\sigma^{-1}$ are primitive recursive. Note that the fully
primitive recursive permutations form a group.

\begin{thm} \label{th:truth table}
Given a $\Pi^0_1$ set $S$ we can effectively build   fully primitive recursive
permutations $\beta, \sss, \ttt$ such that the group $G$ generated by them
has word problem in the same truth-table degree as~$S$.
\end{thm}
\begin{proof}
In this proof we   work with  an array of  columns indexed by integers.  Let
$\sss (\la x, n \ra) = \la x+1, n \ra$  ($x \in \mathds{Z}, n \in \omega$)
be the shift to the next column. Let $\tau$ consist of the 2-cycles  $(\la 0,
3t +1 \ra ,  \la 0, 3t+2 \ra)$ for each $t$: in other words, $\tau(\la 0,
3t +1 \ra) =\la 0, 3t+2 \ra$, $\tau(\la 0,
3t +2 \ra) =\la 0, 3t+1 \ra$ for all $t$, and $\tau$ is the identity elsewhere.

Let $S$ be a given $\Pi^{0}_{1}$ set, and let $S^{c}= \omega \smallsetminus
S$ be the complement of $S$. First we show we may assume that, up to
$m$-equivalence,  $S^{c}$  is the range of a  $1$-$1$ function with graph
effectively given  by an index for a  primitive recursive relation.  We can
uniformly replace $S^{c}$ by $\{2n \colon \, n \in  S^{c}\}  \cup \{2n+1
\colon n \in \omega\}$, so we may assume that $S^{c}$ is infinite. From a
c.e.\ index for $S^{c}$ we may effectively obtain an index $e$  of a Turing
machine  that computes a  $1$-$1$   function $f$ with range $S^{c}$. Thus,
for all $x$ we have $f(x)=U(\mu y. \,T(e,x,y))$, where $U$ and $T$ are
respectively a primitive recursive function and a primitive recursive
predicate as in the Kleene Normal Form Theorem. Consider the primitive
recursive predicate  $P(e,x,y)$, which holds if and only if $T(e,x,y) \,\&\,
\forall z < y \, [\neg T(e,x,z)]$.  Using  the standard primitive recursive
pairing function $\la . \,, . \ra$, let $g (\la x, y \ra) = 2 U(y)$ if
$P(e,x,y)$ holds, and  $g (\la x, y \ra) = 2 \la x, y \ra + 1$ otherwise.
Clearly $g$ is a $1$-$1$ function with primitive recursive graph. The range
of $g$ is $\{2n \colon \, n \in  S^{c}\} \cup \{2 \la x, y \ra + 1\colon \,
\neg P(e,x,y)\}$, which is   $m$-equivalent to $S^{c}$.

Next we code the graph of $g$  into a fully primitive recursive permutation
$\beta$ as follows:  if $g(t) = x$, then $\beta $ has  a 2-cycle $(\la x, 3t
\ra, \la x, 3t+1 \ra)$. Thus, among the three permutations only $\beta$
depends on $S$.  Clearly $\beta $ is fully primitive recursive uniformly in a
c.e.\ index for $S^c$.

Let $G$ be the group of   permutations generated by $\sigma, \tau, \beta$.
For $x \in \omega$, we can picture $\Cj (\beta , \sss^{-x})$ as the
``shift'' of $\beta$ by $x$ columns to the left. The set $S$  is many-one
below the word problem of $G$ because
\[
x \in S \Leftrightarrow  [\Cj (\beta , \sss^{-x}), \ttt] =1,
\]
where $[u,v]=u^{-1}v^{-1}uv$ is the usual commutator of $u,v$. To see this,
first note that if $y \ne 0$, then $\Cj (\beta , \sss^{-x})(\langle
y,t\rangle)$ still lies in the $y$-th column, and thus $\Cj (\beta ,
\sss^{-x}) \tau(\langle y,t\rangle)=\tau \Cj (\beta , \sss^{-x})(\langle
y,t\rangle$), as $\tau$ is the identity on the $y$-th column. Now, if $x \in
S$, then $\beta$ is the identity on the $x$-th column and thus  $\Cj (\beta ,
\sss^{-x})$ is the identity on the $0$-th column, giving $[\Cj (\beta ,
\sss^{-x}), \ttt] =1$; if $x \notin S$, and $t$ is such that $g(t)=x$, then
$\Cj (\beta, \sss^{-x}) \tau(\langle 0 ,i\rangle) \ne \tau\Cj (\beta ,
\sss^{-x})(\langle 0,i\rangle)$, for every $i \in \{3t,3t+1,3t+2\}$.

It remains to show that  the  word problem of $G$ is truth-table below $S$.
We note that $\tau$ and $\beta$ are involutions. For any  $x \in \mathds{Z}$
we write $\beta_x = \Cj (\beta , \sss^{-x})$ and $\tau_x= \Cj(\tau,
\sss^{-x})$.  It is easy to see that  $[\beta_x, \beta_y]=1$ and $[\tau_x,
\tau_y]=1$, for all $x,y$.
Suppose now that a word $w\in F(\beta, \sigma, \tau)$ (the free group on
$\{\beta, \sigma, \tau\}$) is given; we have to decide whether $w=1$ in $G$
by  accessing the oracle $S$ in a truth-table fashion. If the exponent sum of
$\sigma$ in $w$ (i.e. the sum of all exponents of occurrences of $\sigma$ in
$w$) is nonzero then $w \neq 1$ in $G$. Otherwise, using the observations
above, we can effectively replace $w$  by an equivalent word
\begin{equation}\label{eqn:w-replace}
(\prod_{x \in L_1} \beta_x) (\prod_{u \in M_1} \tau_{u}) (\prod_{x \in L_2} \beta_x)
(\prod_{u \in M_2} \tau_{u}) \ldots   (\prod_{x \in L_k} \beta_x) (\prod_{u \in M_k} \tau_{u})
\end{equation}
where the the $L_i$ and $M_i$ are effectively
given  finite sets of distinct integers, which are nonempty except for possibly $L_1$
or $M_k$. Let $L=\bigcup_{i} L_{i}$ and $M=\bigcup_{i} M_{i}$.

Notice that a product $\beta_x \tau_u$ produces a $3$-cycle in column $-u$
precisely when $x -u\in  S^c$, otherwise  $\beta_x \tau_u$ coincides on
$C_{-u}$ with $\tau_u$. For every $x,u$ let $w(x,u)$ be the word obtained
from (\ref{eqn:w-replace}) by deleting all elements different from $\beta_x,
\tau_u$, and cancelling all occurrences of subwords $\beta_x \beta_x$ and
$\tau_u\tau_u$. Since $g$ is $1$-$1$, we have that the cycles of $\beta_x$
and $\beta_y$ are disjoint for any $x \neq y$: therefore the permutations
corresponding to $w(x,u)$ and $w$ coincide in the interval  $\{\langle -u, 3t
\rangle, \langle -u, 3t+1 \rangle, \langle -u, 3t+2 \rangle\}$ of the column
$C_{-u}$, where $g(t)=x$.

To decide whether the word in (\ref{eqn:w-replace}) is equal to $1$ in $G$,
we give a procedure to decide whether the permutation corresponding to $w$ is
the identity on each column $C_{-u}$.
 First notice that $w$ fixes all columns $C_{-u}$ with $u \notin M$
if and only, for all $x \in L$, the number of occurrences of $\beta_x$ in
(\ref{eqn:w-replace}) is even. Indeed, if $u \notin M$ and $x \in L$, then
$w(x,u)$ is a word consisting of only occurrences of $\beta_x$, which by
cancellation is either empty (if the number of occurrences is even) or equal
to $\beta_x$: if the former case happens for every $x \in L$, then every
column $C_{-u}$ with $u\notin M$ remains fixed; if $x\in L$ satisfies the
latter case, and $u \notin M$ is such that $x-u\in S^c$, then $w$ does not
fix $C_{-u}$, in which case we output $w \ne 1$ in $G$.

If we have already ascertained that all columns $C_{-u}$ remain fixed for all
$u \notin M$, then take any $u \in M$, and for every $x \in L$, perform the
following check querying the oracle:
\begin{enumerate}
\item if $x-u \notin S^c$ then on the column $C_{-u}$ the permutation
    corresponding to $w(x,u)$  coincides with the one corresponding to the word
    obtained from it by cancelling all occurrences of $\beta_x$; in this
    case, state that $C_{-u}$ is \emph{$x$-fixed} if and only if the length
    of the resulting word is even;
\item if $x-u \in S^c$ and the number of occurrences in $w(x,u)$ of the
    subword $\beta_x \tau_u$ is a not a multiple of $3$, then the
    $3$-cycles produced by $\beta_{x}$ and $\tau_{u}$ do not cancel each
    other: state in this case that $C_{-u}$ is \emph{not $x$-fixed};
    otherwise, cancel from $w(x,u)$ all occurrences of $\beta_x\tau_u$, and
    state that $C_{-u}$ is \emph{$x$-fixed} if and only if the resulting word
    is empty.
\end{enumerate}
If for all $x\in L$ we have stated that $C_{-u}$ is $x$-fixed, then we conclude that
$C_{-u}$ is fixed under the permutation corresponding to $w$.

If for all $u \in M$, we have concluded that $C_{-u}$ is fixed, then we
output that $w=1$ in $G$; otherwise we output $w \neq 1$ in $G$. An output
will be  achieved no matter what the oracle is, so the reduction is
truth-table.
\end{proof}

It would be interesting to determine the complexity of isomorphism and
embedding for f.g.\ groups of recursive permutations. Totality of a function
described  by a recursive index is already $\Pi^0_2$ complete, so  it might
be  more natural to restrict oneself to   fully primitive recursive
permutations as defined above.  It    is a $\Pi^0_1$ condition of an index
consisting of a   pair of indices  $(e,i)$ for primitive recursive functions
(one for the potential permutation, one for its potential inverse) whether it
describes such a permutation.

In both settings, isomorphism and embedding are  $\Sigma^0_3$ relations
between finitely generated groups given by finite sets of indices for the
generators.  For an example where the isomorphism relation has an
intermediate complexity, suppose the domain is $\mathds Z$, and consider the
subgroup $G$ of the group of computable permutations generated by the shift.
The problem whether a   group generated by finitely many fully primitive
recursive permutations is isomorphic to  $G$ is $\Pi^0_2$-hard. To see this,
note that   infinity of a c.e.\ set $W_e$  is $\Pi^0_2$-complete.  Build a
fully primitive recursive permutation  $p_e$ by adding a cycle of length $n$
involving large numbers when  $n $ enters $ W_e$. Then the subgroup generated
by $p_e$ is isomorphic to $G$ if and only if $p_e$  has infinite order, if and
only if  $W_e$ is infinite.

%\bibliographystyle{plain}
%\bibliography{wordproblems}

\end{document}